\documentclass[reqno, 12pt]{article}

\pdfoutput=1

\usepackage{enumerate}
\usepackage{latexsym}
\usepackage[centertags]{amsmath}
\usepackage{amsfonts}
\usepackage{amssymb}
\usepackage{amsthm}
\usepackage{newlfont}
\usepackage{graphics}
\usepackage{color}
\usepackage{float}
\usepackage{diagbox}
\usepackage[linesnumbered,ruled,vlined]{algorithm2e}
\usepackage[T1]{fontenc}
\usepackage{lmodern}
\usepackage{microtype}
\usepackage{longtable}
\usepackage{rotating}
\usepackage{multirow}
\usepackage{extarrows}
\usepackage[sort,compress,numbers]{natbib}
\usepackage[utf8]{inputenc}
\usepackage{xcolor}
\usepackage{listings}
\usepackage{aliascnt}
\usepackage[pagebackref,colorlinks=true,linkcolor=blue,citecolor=red,urlcolor=blue]{hyperref}
\usepackage[nameinlink,noabbrev,capitalize]{cleveref}

\newtheorem{thm}{Theorem}[section]
\newtheorem{cor}[thm]{Corollary}
\newtheorem{lem}[thm]{Lemma}
\newtheorem{prop}[thm]{Proposition}

\theoremstyle{mydefinition}

\theoremstyle{myremark}

\newtheorem{exa}[thm]{Example}
\newtheorem{Openprob}[thm]{Open problem}

\allowdisplaybreaks[4]
\SetKwInput{KwInput}{Input}                
\SetKwInput{KwOutput}{Output}              

\renewcommand{\P}{{\mathbb{P}}}

\title{Hankel Determinants for Convolution of Power Series: An Extension of Cigler's Results}

\author{Feihu Liu$^{\color{blue} \dag}$, Ying Wang$^{\color{blue} \ddag}$, Yingrui Zhang$^{\color{blue} \S}$, and Zihao Zhang$^{\color{blue} \P}$
\\[2mm]
{\small $^{\color{blue} \dag}$ School of Mathematical Sciences,}\\[-0.8ex]
{\small Capital Normal University, Beijing, 100048, P.R.~China}\\
{\small $^{\color{blue} \ddag}$ School of Mathematics and Statistics,}\\[-0.8ex]
{\small North China University of Water Resources and Electric Power, Zhengzhou, 450045, P.R.~China}\\
{\small $^{\color{blue} \S}$ School of Statistics and Mathematics,}\\[-0.8ex]
{\small Yunnan University of Finance and Economics, Kunming, 650221, P.R.~China}\\
{\small $^{\color{blue} \P}$ School of Mathematics and Statistics,}\\[-0.8ex]
{\small Beijing Institute of Technology, Beijing 102400, P.R.~China}\\
{\small {\color{blue} $^\dag$} Email address: liufeihu7476@163.com}\\
{\small {\color{blue} $^\ddag$} Email address: wangying2019@ncwu.edu.cn}\\
{\small {\color{blue} $^\S$} Email address: zyrzuhe@126.com}\\
{\small {\color{blue} $^\P$} Email address: zihao-zhang@foxmail.com}
}

\date{\today}

\begin{document}

\maketitle

\begin{abstract}
Cigler considered certain shifted Hankel determinants of convolution powers of Catalan numbers and conjectured identities for these determinants. Recently, Fulmek gave a bijective proof of Cigler's conjecture. Cigler then provided a computational proof.
We extend Cigler's determinant identities to the convolution of general power series $F(x)$, where $F(x)$ satisfies a certain type of quadratic functional equation. As an application, we present the Hankel determinant identities of convolution powers of Motzkin numbers.
\end{abstract}

\noindent
\begin{small}
\end{small}

\noindent
\begin{small}
\end{small}

\section{Introduction}
Let $F=(a_0,a_1,a_2,\ldots)$ be a sequence with $a_i\in \mathbb{R}$, and denote by
\begin{align*}
F(x)=a_0+a_1x+a_2x^2+\cdots \in \mathbb{R}[[x]]
\end{align*}
its generating function.
Define the \emph{Hankel matrices} of order $n$ of $F(x)$ by
\begin{align*}
\mathcal{H}_n(F(x))= (a_{i+j})_{0\leq i,j\leq n-1}
=\left ( \begin{matrix}
a_0 & a_1 & \cdots & a_{n-1} \\
a_1 & a_2 & \cdots & a_{n}  \\
\vdots & \vdots & \ddots & \vdots \\
a_{n-1} & a_{n} & \cdots & a_{2n-2} \\
\end{matrix} \right )
\end{align*}
or abbreviated as $\mathcal{H}_n(F)$.
The \emph{Hankel determinant} of order $n$ of $F(x)$ is defined by $\det \mathcal{H}_n(F(x))$ (or $\det \mathcal{H}_n(F)$). We set $\det \mathcal{H}_0(F)=1$ by convention.

The computation of Hankel determinants is a well-studied topic in combinatorics and number theory. For instance, Hankel determinants for specific sequences have been explored in several works \cite{chien2022hankel,Elouafi,Mu-Wang,Mu-Wang-Yeh,QH-Hou}, among others. Many continued fraction methods for calculating Hankel determinants have been developed and studied. For example, these include the S-continued fraction (or J-continued fraction) \cite{Krattenthaler05, FlajoletDM}, Gessel-Xin's continued fraction method \cite{Gessel-Xin}, Sulanke-Xin's continued fraction method \cite{Sulanke-Xin}, the Hankel continued fraction \cite{HanGuoNiu}, and the shifted periodic continued fraction \cite{Wang-Xin-Zhai}.

The \emph{$k$-th convolution of generating function $F(x)$} is defined by
$$F(x)^k=\Big( \sum_{n\geq 0}a_n x^n\Big)^k=\sum_{n\geq 0}a_{k,n}x^n.$$
Inspired by Cigler's work \cite{Cigler-2403}, we study the following shifted Hankel determinants of $F(x)^k$:
$$D_{K,M}(N):=\det (a_{K,i+j+M})_{i,j=0}^{N-1},$$
where $M\in \mathbb{Z}$, and $K,N\in \mathbb{N}$, $K,N\geq 1$. We set $D_{K,M}(0)=1$ by convention.

When
$$F(x)=\sum_{n\geq 0}c_nx^n=\frac{1-\sqrt{1-4x}}{2x}$$
is the generating function of the \emph{Catalan numbers} $c_n=\frac{1}{n+1}\binom{2n}{n}$, Cigler \cite{Cigler-2306,Cigler-2308} conjectured the identities of the Hankel determinants below (Theorem \ref{Cigler-Catalan}).

\begin{thm}{\em (Cigler's Conjecture)}\label{Cigler-Catalan}
Let $F(x)$ be the generating function of Catalan numbers. Let $k$ be a positive integer and $m\in \mathbb{N}$.
Then for $K=2k$ we have the even identities
\begin{align*}
&D_{2k,1-k-m}(N)=0\ \ \ \text{for}\ \ \ N=1,2,\ldots,m+k-1,
\\& D_{2k,1-k-m}(n+m+k)=(-1)^{\binom{m+k}{2}} D_{2k,1-k+m}(n)\ \ \ \ \text{for all}\ n\in \mathbb{N}.
\end{align*}
For $K=2k-1$ we have the odd identities
\begin{align*}
&D_{2k-1,2-k-m}(N)=0\ \ \ \text{for}\ \ \ N=1,2,\ldots,m+k-2,
\\& D_{2k-1,2-k-m}(n+m+k-1)=(-1)^{\binom{m+k-1}{2}} D_{2k-1,1-k+m}(n)\ \ \ \ \text{for all}\ n\in \mathbb{N}.
\end{align*}
\end{thm}

Recently Fulmek \cite{Fulmek-2402} gave a bijective proof of Cigler's conjecture by interpreting determinants as generating functions of nonintersecting lattice paths. Fulmek's proof employs the reflection principle and the Lindstr\"om-Gessel-Viennot lemma \cite{GesselViennot,Lindstrom}. Subsequently, Cigler \cite{Cigler-2403} provided an elegant and short proof. Cigler's proof is based on earlier results of Cigler \cite[Proposition 2.5]{Cigler-1302} and Andrews and Wimp \cite{Andrews-Wimp}.

It is well known that the generating function of Catalan numbers satisfies the quadratic functional equation: $1-F(x)+xF(x)^2=0$.
When $F(x)$ is a power series on a finite field and satisfies a certain type of quadratic functional equation, Han \cite{HanGuoNiu} proved that the Hankel continued fraction expansion of $F(x)$ exists and is ultimately periodic. Moreover, the Hankel determinant sequence $(\det \mathcal{H}_n(F(x)))_{n\geq 0}$ is ultimately periodic.

In this paper, we first establish an identity about the $k$-th convolution of $F(x)$, see Proposition \ref{FK-convlution}. Then we generalize Cigler's results (Theorem \ref{Cigler-Catalan}) to a general power series $F(x)$, where $F(x)$ satisfies a certain type of quadratic functional equation, see Theorem \ref{Main-Theorem}. Finally, we apply this result to generating functions of Catalan numbers, Motzkin numbers, etc.
We would also like to mention that Cigler has already considered a quadratic family, namely the Narayana polynomials, in \cite[Section 3]{Cigler-2403}.

When we study the concrete value of $D_{K,M}(N)$ for the convolution of the generating function of Motzkin numbers, we find that $(D_{K,M}(n))_{n\geq 0}$ is shifted periodically. The meaning of shifted periodically is that, when $n$ lies in a fixed residue class modulo certain integer $q$, the value of determinant is a linear function of $\lfloor n/q\rfloor$, see Example \ref{Example-D33} and Proposition \ref{Prop-Mozt-123}. This inspires us to study the exact formulas of $D_{K,M}(n)$ when $K\in \mathbb{N}$ and $M\in \mathbb{Z}$ are concrete values. As byproducts, we obtain some precise formulas through Sulanke-Xin's quadratic transformation \cite{Sulanke-Xin}. For example, $D_{3,-3}(n)$, $D_{2,-3}(n)$, and so on (see Propositions \ref{Prop-D33-proof} and \ref{Prop-Mozt-123}). More formulas can be obtained using similar techniques.

This paper is organized as follows. In Section~\ref{Section-2}, we present our main results.
In Section~\ref{Section-3}, we research the applications of our results on Catalan numbers and Motzkin numbers.
In Section~\ref{Section-4}, we discuss the exact formulas of $D_{K,M}(n)$ when $K\in \mathbb{N}$ and $M\in \mathbb{Z}$ are concrete values. Finally, we propose an open problem in Section~\ref{Section-5}.

\section{Main Results}\label{Section-2}

The main result of this section is Theorem \ref{Main-Theorem}, but we first prove the preliminary Lemma \ref{Cigler-ST} and Proposition \ref{FK-convlution}, which are necessary for its proof.

For an arbitrary two variable generating function $D(x, y) =\sum_{i,j=0}^{\infty} d_{i,j}x^iy^j$, let
$[D(x, y)]_n$ be the determinant of the $n \times n$ matrix $(d_{i,j})_{0\leq i,j\leq n-1}$. It may be intuitive to
see that the Hankel determinants, $\det \mathcal{H}_n(F)$ can be expressed as
\begin{align}\label{Definition-Square-oper}
\det \mathcal{H}_n(F)= \left[\frac{xF(x)-yF(y)}{x-y}\right]_n.
\end{align}
This follows from
$$\frac{xF(x)-yF(y)}{x-y}=\frac{x\sum_{i\geq 0}a_ix^i-y\sum_{i\geq 0}a_iy^i}{x-y}=\sum_{i\geq 0}a_i(x^i+x^{i-1}y+\cdots+y^i).$$

There are three simple rules to transform the determinant $[D(x, y)]_n$ to another determinant, see \cite{Gessel-Xin}.
In this paper, we only use the following \emph{product rules}:
If $u(x)$ is any formal power series with $u(0) =1$, then $[u(x)D(x, y)]_n =[D(x, y)]_n=[u(y)D(x, y)]_n$. This rule corresponds to sequences of elementary row or column operations.

The following lemma appears in \cite[Proposition 2.5]{Cigler-1302}, which is a slightly generalized result of Andrews and Wimp \cite{Andrews-Wimp}. For the sake of completeness, we give another simple proof.

\begin{lem}\label{Cigler-ST}
Let $s(x)=\sum_{n\geq 0}s_nx^n$ with $s_0=1$ and
$$t(x)=\frac{1}{s(x)}=\sum_{n\geq 0}t_nx^n.$$
Setting $s_n=t_n=0$ for $n<0$. For $M,N\in \mathbb{N}$, we get
$$\det(s_{i+j-M})_{i,j=0}^{N+M}=(-1)^{N+\binom{M+1}{2}}\det(t_{i+j+M+2})_{i,j=0}^{N-1}.$$
\end{lem}
\begin{proof}
Since $s_0=1$, $t_0=1$. Let $T_M(x)=1+t_1x+\cdots +t_{M}x^{M}$. Then we have
\begin{align*}
\det(s_{i+j-M})_{i,j=0}^{N+M}&=\left[\frac{x^{M+1}s(x)-y^{M+1}s(y)}{x-y}\right]_{N+M+1}
=\left[\frac{\frac{x^{M+1}}{t(x)}-\frac{y^{M+1}}{t(y)}}{x-y} \times t(x)t(y)\right]_{N+M+1}
\\&=\left[\frac{x^{M+1}t(y)-y^{M+1}t(x)}{x-y}\right]_{N+M+1}.
\end{align*}
The first equality is due to Equation \eqref{Definition-Square-oper}. The second equality is obtained by applying the product rules of $[\ \cdot\ ]_n$.
Then, standard arguments yield
\begin{align*}
&\det(s_{i+j-M})_{i,j=0}^{N+M}
\\=&\left[\frac{T_M(y)x^{M+1}-T_M(x)y^{M+1}}{x-y}+\frac{(t(y)-T_M(y))x^{M+1}-(t(x)-T_M(x))y^{M+1}}{x-y}\right]_{N+M+1}
\\=&\Bigg[\frac{(x^{M+1}-y^{M+1})+t_1xy(x^{M}-y^{M})+\cdots+t_{M}(xy)^{M}(x-y)}{x-y}
\\&\quad\quad\quad\quad\quad +(xy)^{M+1}\frac{\frac{(t(y)-T_M(y))}{y^{M+1}}-\frac{(t(x)-T_M(x))}{x^{M+1}}}{x-y}\Bigg]_{N+M+1}.
\end{align*}
According to the definition of $[\ \cdot\ ]_n$, it follows that
\begin{align*}
&\det(s_{i+j-M})_{i,j=0}^{N+M}=\det\left(
\begin{array}{cc}
A & 0\\
0& B\\
\end{array}
\right)_{(N+M+1)\times (N+M+1)},
\end{align*}
where
\begin{align*}
A=\left(
\begin{array}{ccccc}
0 & 0&  \cdots & 0 & 1 \\
0 & 0 & \cdots & 1 & t_1 \\
\vdots & \vdots & \ddots & \vdots & \vdots \\
0 & 1 & \cdots & t_{M-2} & t_{M-1} \\
1 & t_1 & \cdots & t_{M-1} & t_{M} \\
\end{array}
\right)_{(M+1)\times (M+1)}
\end{align*}
and
\begin{align*}
B=\left(
\begin{array}{cccc}
-t_{M+2} & -t_{M+3} & \cdots  & -t_{M+N+1} \\
-t_{M+3}  & -t_{M+4} &\cdots  & -t_{M+N+2} \\
\vdots & \vdots & \ddots & \vdots \\
-t_{M+N+1} & -t_{M+N+2} & \cdots & -t_{M+2N} \\
\end{array}
\right)_{N\times N}.
\end{align*}
It is clear that
\begin{align*}
\det A=(-1)^{\binom{M+1}{2}}\ \ \text{and }\ \ \det B=(-1)^N\det(t_{i+j+M+2})_{i,j=0}^{N-1}.
\end{align*}
This completes the proof.
\end{proof}

Similarly, the reader can verify that Lemma \ref{Cigler-ST} also holds for $M=-1$.

Let $T(k,i)=\frac{(k-i-1)!\cdot k}{i!\cdot (k-2i)!}$. We found that $T(k,i)$, $k>0, 0\leq i\leq\lfloor\frac{k}{2}\rfloor$ is the triangle of coefficients of Lucas polynomials, see \cite[A034807]{Sloane23}. It is easy to verify that $T(k,i)$ satisfies the following recursive formula:
\begin{align}
&T(k,i)=T(k-1,i)+T(k-2,i-1),\quad k>1;\qquad T(k,0)=1,\quad k>0.\label{Rec-Labe-Tki}
\\& T(0,0)=2,\qquad T(k,i)=0 \quad\text{if } i<0\text{ or } 2i>k\ge0.\notag
\end{align}

\begin{prop}\label{FK-convlution}
Let $F(x)$ be a power series satisfying the following quadratic functional equation
\begin{align}\label{Quadratic-FE}
w(x)+u(x)F(x)+v(x)F(x)^2=0,
\end{align}
where $w(x),u(x),v(x)$ are three polynomials with $v(0)=0$, $u(0)\neq 0$.
Then for $k\geq 1$ we have
\begin{align}\label{FXK-convol-Equa}
F(x)^k=\frac{-w(x)^k}{F(x)^kv(x)^k+\sum_{i=0}^{\lfloor k/2\rfloor}(-1)^{k+i+1}T(k,i)u(x)^{k-2i}v(x)^iw(x)^i},
\end{align}
where $T(k,i)=\frac{(k-i-1)!\cdot k}{i!\cdot (k-2i)!}$.
\end{prop}

Note that conditions $v(0)=0$ and $u(0)\neq 0$ ensure the uniqueness of the power series $F(x)$. Readers can refer to \cite{HanGuoNiu}.

For convenience, we define
\begin{align}\label{LKX-Convenien}
L_k(x)=\sum_{i=0}^{\lfloor k/2\rfloor}(-1)^{k+i}T(k,i)u(x)^{k-2i}v(x)^iw(x)^i.
\end{align}
For example, the first polynomials $L_k(x)$ are
\begin{align*}
&L_1(x)=-u(x),\ \ \ L_2(x)=u(x)^2-2v(x)w(x),\ \ \ L_3(x)=-u(x)^3+3u(x)v(x)w(x),
\\& L_4(x)=u(x)^4-4u(x)^2v(x)w(x)+2v(x)^2w(x)^2,
\\& L_5(x)=-u(x)^5+5u(x)^3v(x)w(x)-5u(x)v(x)^2w(x)^2.
\end{align*}

\begin{proof}[First proof of Proposition \ref{FK-convlution}]
We found that Equation \eqref{FXK-convol-Equa} is equivalent to
\begin{align}\label{Lemm-F-F-L}
F(x)^kv(x)^k+\frac{w(x)^k}{F(x)^k}=L_k(x).
\end{align}
We prove the identity by showing that both sides satisfy the same recurrence relation and initial conditions.
Firstly, we verify that $L_k(x)$ satisfies the recursion formula
$$L_{k+1}(x)=-u(x)L_k(x)-v(x)w(x)L_{k-1}(x).$$
We infer that
\begin{align*}
L_{k+1}(x)&=\sum_{i=0}^{\lfloor (k+1)/2\rfloor}(-1)^{k+1+i}T(k+1,i)u(x)^{k+1-2i}v(x)^iw(x)^i\quad \text{(By Equation \eqref{LKX-Convenien})}
\\&=\sum_{i=0}^{\lfloor (k+1)/2\rfloor}(-1)^{k+1+i}(T(k,i)+T(k-1,i-1))u(x)^{k+1-2i}v(x)^iw(x)^i\ \text{(By Equation \eqref{Rec-Labe-Tki})}
\\&=\sum_{i=0}^{\lfloor (k)/2\rfloor}(-1)^{k+1+i}T(k,i)u(x)^{k+1-2i}v(x)^iw(x)^i
\\&\ \ \ \ \ +\sum_{i=1}^{\lfloor (k+1)/2\rfloor}(-1)^{k+1+i}T(k-1,i-1)u(x)^{k+1-2i}v(x)^iw(x)^i
\\&=-u(x)L_k(x)+\sum_{i=0}^{\lfloor (k-1)/2\rfloor}(-1)^{k+i+2}T(k-1,i)u(x)^{k-1-2i}v(x)^{i+1}w(x)^{i+1}
\\&=-u(x)L_k(x)-v(x)w(x)L_{k-1}(x).
\end{align*}
The last and the second last equalities are again derived from Equation \eqref{LKX-Convenien}.
Secondly, let us define
$$G_k(x):=F(x)^kv(x)^k+\frac{w(x)^k}{F(x)^k}.$$
Then we obtain
\begin{align*}
&G_{k+1}(x)=F(x)^{k-1}v(x)^k(F(x)^2v(x))+\frac{w(x)^k}{F(x)^{k-1}}\left(\frac{w(x)}{F(x)^2}\right)
\\=&F(x)^{k-1}v(x)^k(-u(x)F(x)-w(x))+\frac{w(x)^k}{F(x)^{k-1}}\left(-v(x)-\frac{u(x)}{F(x)}\right) \quad \text{(By Equation \eqref{Quadratic-FE})}
\\=&-u(x)\left(F(x)^kv(x)^k+\frac{w(x)^k}{F(x)^k}\right)-v(x)w(x)\left(F(x)^{k-1}v(x)^{k-1}+\frac{w(x)^{k-1}}{F(x)^{k-1}}\right)
\\=&-u(x)G_k(x)-v(x)w(x)G_{k-1}(x).
\end{align*}
That is, $L_k(x)$ and $G_k(x)$ satisfy the same recursion formula.
Finally, it is easily verified that
$L_1(x)=G_1(x)=-u(x)$ and
$$L_2(x)=G_2(x)=u(x)^2-2v(x)w(x).$$
This completes the proof.
\end{proof}

For readers familiar with symmetric functions \cite[Chapter 7]{RP.Stanley99}, we provide the following proof.
\begin{proof}[Second proof of Proposition \ref{FK-convlution}]
Similarly, our task is to prove
\begin{align*}
F(x)^kv(x)^k+\frac{w(x)^k}{F(x)^k}=L_k(x).
\end{align*}
Let $X:=F(x)v(x)$ and $Y:=\frac{w(x)}{F(x)}$. The problem was reduced to proving $X^k+Y^k=L_k(x)$.
It can easily be shown that $X+Y=-u(x)=L_1(x)$ and $XY=v(x)w(x)$.
We immediately obtain $X^2+Y^2=(X+Y)^2-2XY=u(x)^2-2v(x)w(x)=L_2(x)$.

Now we define $p_k(X,Y)=X^k+Y^k$ and $e_1(X,Y)=X+Y$, $e_2(X,Y)=XY$.
(In fact, $p_k$ and $e_k$ are respectively called power sum symmetric functions and elementary symmetric functions \cite[Chapter 7]{RP.Stanley99}.)
By applying the well-known Waring's formula \cite[Page 434]{RP.Stanley99}, we obtain
$$\sum_{k\geq 1}\frac{(-1)^{k-1}p_k(X,Y)}{k}t^k=\log (1+e_1(X,Y)t+e_2(X,Y)t^2).$$
Therefore, we have
\begin{align*}
\sum_{k\geq 1}\frac{(-1)^{k-1}p_k(X,Y)}{k}t^k&=\sum_{n\geq 1}\frac{(-1)^{n-1}}{n}(e_1(X,Y)t+e_2(X,Y)t^2)^n
\\&=\sum_{n\geq 1}\frac{(-1)^{n-1}}{n}\sum_{i=0}^n\binom{n}{i}e_1(X,Y)^{n-i}e_2(X,Y)^it^{n+i}.
\end{align*}
By taking the coefficients of $t^k$ on both sides of the equation, we obtain
$$\frac{(-1)^{k-1}p_k(X,Y)}{k}=\sum_{i=0}^{\lfloor k/2\rfloor}\frac{(-1)^{k-i-1}}{k-i}\binom{k-i}{i}e_1(X,Y)^{k-2i}e_2(X,Y)^i.$$
By combining \eqref{LKX-Convenien} with $e_1(X,Y)=X+Y=-u(x)$ and $e_2(X,Y)=XY=v(x)w(x)$, this completes the proof.
\end{proof}

In fact, $L_k$ can be represented by Dickson polynomials. The \emph{Dickson polynomial} $D_k(z,q)$ is defined by 
$$D_0(z,q)=2,\quad D_1(z,q)=z,\quad D_{k+1}(z,q)=zD_k(z,q)-qD_{k-1}(z,q).$$
If the reader is familiar with Dickson polynomials, they will observe that $L_k=D_k(-u,vw)$.
We also independently discovered Proposition \ref{FK-convlution}.

The following theorem extends Cigler's result (Theorem \ref{Cigler-Catalan}).
\begin{thm}\label{Main-Theorem}
Let $k$ be a positive integer and $m\in \mathbb{N}$.
Let $F(x)\in \mathbb{R}[[x]]$ be a power series satisfying the following quadratic functional equation
$$1+u(x)F(x)+x^aF(x)^2=0,$$
where $u(x)$ is a polynomial with $u(0)=-1$ and $a\geq 1$.
Now by Equation \eqref{LKX-Convenien},
$$L_k(x)=\sum_{i=0}^{\lfloor k/2\rfloor}(-1)^{k+i}T(k,i)u(x)^{k-2i}x^{ai}.$$
Let $\deg L_{2k}(x)=b^{\prime}$.
If $ka\geq b^{\prime}$, then we have
\begin{align}
&D_{2k,1-m-ak}(N)=0\ \ \ \text{for}\ \ \ N=1,2,\ldots,m+ak-1,\label{Equ-Even-N}
\\& D_{2k,1-m-ak}(n+m+ak)=(-1)^{\binom{m+ak}{2}} D_{2k,m+1-ak}(n)\ \ \ \ \text{for all}\ n\in \mathbb{N}.\label{Equ-Even-Rev}
\end{align}
Let $\deg L_{2k-1}(x)=b^{\prime\prime}$.
If $ka-1\geq b^{\prime\prime}$, then we have
\begin{align}
&D_{2k-1,2-m-ak}(N)=0\ \ \ \text{for}\ \ \ N=1,2,\ldots,m+ak-2,\label{Equ-Odd-N}
\\& D_{2k-1,2-m-ak}(n+m+ak-1)=(-1)^{\binom{m+ak-1}{2}} D_{2k-1,m+a-ak}(n)\ \ \ \ \text{for all}\ n\in \mathbb{N}.\label{Equ-Odd-Rev}
\end{align}
\end{thm}
\begin{proof}
Identities \eqref{Equ-Even-N} and \eqref{Equ-Odd-N} can be easily obtained as the first rows of the Hankel matrices vanish.
Since $u(0)=-1$, we have $F(0)=1$. We may use Lemma \ref{Cigler-ST} to prove the identities \eqref{Equ-Even-Rev} and \eqref{Equ-Odd-Rev}.

To prove \eqref{Equ-Even-Rev}, we choose $s(x)=F(x)^{2k}$, $M=m+ak-1$, and $N=n$. Then we obtain
$$\det(s_{i+j-M})_{i,j=0}^{N+M}=\det(s_{i+j-m-ak+1})_{i,j=0}^{n+m+ak-1}=D_{2k,1-m-ak}(n+m+ak).$$
By \eqref{Lemm-F-F-L}, we get $t(x)=L_{2k}(x)-(x^aF(x))^{2k}$. Note that $L_{2k}(x)$ is a polynomial of degree $b^{\prime}$ in $x$.
Therefore, we have i): $t_{2ka+r}=-a_{2k,r}$ for $r\geq 0$; ii): $t_n=a_{2k,n-2ka}=0$ for $b^{\prime}<n<2ka$.
Thus, we get $t_{r+m+1+b^{\prime}}=-a_{2k,r+m+1+b^{\prime}-2ka}$ for $r+m+1>0$.
This gives
\begin{align*}
&\det(t_{i+j+M+2})_{i,j=0}^{N-1}=\det(t_{i+j+m+ak+1})_{i,j=0}^{n-1}
\\=&\det(-a_{2k,i+j+m+1+ak-2ka})_{i,j=0}^{n-1}\ \ \ (\text{By }\ ka\geq b^{\prime})
\\=&(-1)^n\det(a_{2k,i+j+m+1-ka})_{i,j=0}^{n-1}=(-1)^nD_{2k,m+1-ka}(n).
\end{align*}
Thus Lemma \ref{Cigler-ST} gives
$$D_{2k,1-m-ak}(n+m+ak)=(-1)^{\binom{m+ak}{2}}D_{2k,m+1-ak}(n).$$

To prove \eqref{Equ-Odd-Rev}, we take $s(x)=F(x)^{2k-1}$, $M=m+ak-2$, and $N=n$. Then we get
$$\det(s_{i+j-M})_{i,j=0}^{N+M}=\det(s_{i+j-m-ak+2})_{i,j=0}^{n+m+ak-2}=D_{2k-1,2-m-ak}(n+m+ak-1).$$
According to \eqref{Lemm-F-F-L}, we obtain $t(x)=L_{2k-1}(x)-(x^aF(x))^{2k-1}$. Note that $L_{2k-1}(x)$ is a polynomial of degree $b^{\prime\prime}$ in $x$.
Therefore, we have i): $t_{(2k-1)a+r}=-a_{2k-1,r}$ for $r\geq 0$; ii): $t_n=a_{2k-1,n-(2k-1)a}=0$ for $b^{\prime\prime}<n<(2k-1)a$.
Thus, we get $t_{r+m+1+b^{\prime\prime}}=-a_{2k-1,r+m+1+b^{\prime\prime}-(2k-1)a}$ for $r+m+1>0$.
This gives
\begin{align*}
&\det(t_{i+j+M+2})_{i,j=0}^{N-1}=\det(t_{i+j+m+1+ak-1})_{i,j=0}^{n-1}
\\=&\det(-a_{2k-1,i+j+m+ak-(2k-1)a})_{i,j=0}^{n-1}\ \ \ (\text{By }\ ak-1\geq b^{\prime\prime})
\\=&(-1)^n\det(a_{2k-1,i+j+m+a-ka})_{i,j=0}^{n-1}=(-1)^nD_{2k-1,m+a-ka}(n).
\end{align*}
Thus Lemma \ref{Cigler-ST} gives
$$D_{2k-1,2-m-ak}(n+m+ak-1)=(-1)^{\binom{m+ak-1}{2}}D_{2k-1,m+a-ak}(n).$$
The proof is completed.
\end{proof}

\section{Applications}\label{Section-3}

\subsection{Catalan Numbers}

Let $F(x)=\frac{1-\sqrt{1-4x}}{2x}$ be the generating function of the Catalan numbers.
Then $a_{k,n}=\frac{k}{n+k}\binom{2n+k-1}{n}$ can be obtained from Lagrange inversion \cite[Chapter 5]{RP.Stanley99}.
At this time, the convolution powers $a_{k,n}$ are the number of all non-negative lattice paths in $\mathbb{N}^2$ with up-steps $U=(1,1)$ and down steps $D=(1,-1)$ from $(0,0)$ to $(2n+k-1,k-1)$, see \cite{Cigler-2403}.

The following results are first proven by Wang and Xin \cite{Wang-Xin}. Wang-Xin's proof employs a hypergeometric sum identity and
the EKHAD Maple package of Zeilberger's creative telescoping \cite{Zeilberger}. Here we provide a simple (and short) proof.

\begin{cor}{\em \cite[Proposition 23]{Wang-Xin}}
Let $F(x)$ be the generating function of the Catalan numbers. Then we have
\begin{align*}
L_k(x)=\sum_{i=0}^{\lfloor k/2\rfloor}(-1)^{i}\frac{(k-i-1)!\cdot k}{i!\cdot (k-2i)!}x^{i}.
\end{align*}
\end{cor}
\begin{proof}
It is well known that $1-F(x)+xF(x)^2=0$. Thus, we get $u(x)=-1$ and $a=1$ in Theorem \ref{Main-Theorem}.
The corollary follows immediately from $T(k,i)=\frac{(k-i-1)!\cdot k}{i!\cdot (k-2i)!}$.
\end{proof}

Now that $\deg L_k(x)=\lfloor \frac{k}{2}\rfloor$, it follows that Theorem \ref{Cigler-Catalan} holds.
As special cases, Theorem \ref{Cigler-Catalan} can derive some results in \cite[Corollary 12,13,15]{cigler2011some}.

Let $a_n$ be the number of lattice paths in $\mathbb{N}^2$ from $(0,0)$ to $(n,0)$ using only steps $H=(1,0)$, $U=(1,1)$, and $D=(2,-1)$. The sequence $(a_n)_{n\geq 0}$ is called \emph{generalized Catalan numbers}. The generating function \cite[A023431]{Sloane23} of $a_n$ is
$$F(x)=\sum_{n\geq 0}a_nx^n=\frac{1-x-\sqrt{1-2x+x^2-4x^3}}{2x^3}.$$
Now the generating function $F(x)$ satisfies
$$1+(x-1)F(x)+x^3F(x)^2=0.$$
This satisfies the conditions of Theorem \ref{Main-Theorem}. Therefore, we can obtain a result similar to Theorem \ref{Cigler-Catalan}. We omit it. A similar argument can be found in the next section.

\subsection{Motzkin Numbers}

Let
$$F(x)=\sum_{n\geq 0}a_nx^n=\frac{1-x-\sqrt{1-2x-3x^2}}{2x^2}$$
be the generating function of the \emph{Motzkin numbers} $a_n=\sum_{i\geq 0}\frac{1}{i+1}\binom{n}{2i}\binom{2i}{i}$.
The Motzkin numbers $a_n$ count the number of lattice paths from $(0,0)$ to $(n,0)$ consisting of up-steps $(1,1)$, down-steps $(1,-1)$ and horizontal-steps $(1,0)$ that never fall below the $x$-axis.
The generating function $F(x)$ satisfies
$$1+(x-1)F(x)+x^2F(x)^2=0.$$
At this time, we have $u(x)=x-1$ and $a=2$. Furthermore, we get
$$L_k(x)=\sum_{i=0}^{\lfloor k/2 \rfloor}(-1)^{k+i} T(k,i)(x-1)^{k-2i}x^{2i}$$
and $\deg L_k(x)=k$.

\begin{cor}\label{Cor-Motzkin}
Let $F(x)$ be the generating function of the Motzkin numbers. Let $k$ be a positive integer and $m\in \mathbb{N}$.
Then for $K=2k$, we have the even identities
\begin{align*}
&D_{2k,1-2k-m}(N)=0\ \ \ \text{for}\ \ \ N=1,2,\ldots,m+2k-1,
\\& D_{2k,1-2k-m}(n+m+2k)=(-1)^{\binom{m+2k}{2}} D_{2k,1-2k+m}(n)\ \ \ \ \text{for all}\ n\in \mathbb{N}.
\end{align*}
For $K=2k-1$ we have the odd identities
\begin{align*}
&D_{2k-1,2-2k-m}(N)=0\ \ \ \text{for}\ \ \ N=1,2,\ldots,m+2k-2,
\\& D_{2k-1,2-2k-m}(n+m+2k-1)=(-1)^{\binom{m+2k-1}{2}} D_{2k-1,2-2k+m}(n)\ \ \ \ \text{for all}\ n\in \mathbb{N}.
\end{align*}
\end{cor}

\begin{exa}\label{Example-D33}
Direct computation yields the following determinant sequences. For $k=2$ and $m=1$, the two pairs illustrate the shift identities in Corollary \ref{Cor-Motzkin}.
\begin{align*}
&(D_{4,-4}(n))_{n\geq 0}=(1,0,0,0,0,1,0,0,-1,-1,0,0,1,0,0,0,0,1,0,0,-1,\ldots),
\\& (D_{4,-2}(n))_{n\geq 0}=(1,0,0,-1,-1,0,0,1,0,0,0,0,1,0,0,-1,-1,0,0,1,0,\ldots),
\\& (D_{3,-3}(n))_{n\geq 0}=(1,0,0,0,1,0,-1,2,0,-2,3,0,-3,4,0,-4,5,0,-5,6,0,\ldots),
\\& (D_{3,-1}(n))_{n\geq 0}=(1,0,-1,2,0,-2,3,0,-3,4,0,-4,5,0,-5,6,0,-6,7,0,-7,\ldots).
\end{align*}
\end{exa}

When $F(x)$ is the generating function of the Motzkin numbers, we found that $L_k(x)$ is closely related to the Chebyshev polynomial of the first kind.
The \emph{Chebyshev polynomial of the first kind} $T_k(x), k\geq 1$ is defined by
\begin{align*}
T_1(x)=x,\ \ \ T_2(x)=2x^2-1,\ \ \ T_k(x)=2xT_{k-1}(x)-T_{k-2}(x)\ \ \text{for }\ k\geq 3.
\end{align*}

\begin{prop}
Let $F(x)$ be the generating function of the Motzkin numbers. Then
\begin{align*}
L_k(x)=2(-x)^k T_k\left(\frac{x-1}{2x}\right)\ \ \text{for }\ k\geq 1.
\end{align*}
\end{prop}
\begin{proof}
Let $Q_k(x)=2(-x)^k T_k\left(\frac{x-1}{2x}\right)$. Then $Q_k(x^{-1})=2(-x^{-1})^k T_k\left(\frac{1-x}{2} \right)$. Since $T_k(x)$ is a polynomial of degree $k$ in $x$, it follows that $Q_k(x)$ is a polynomial in $x$.
For $k\geq 2$, we have
\begin{align*}
Q_{k+1}(x)&=2(-x)^{k+1}T_{k+1}\left(\frac{x-1}{2x}\right)=2(-x)^{k+1}\left(2\cdot\frac{x-1}{2x}T_k\left(\frac{x-1}{2x}\right)
-T_{k-1}\left(\frac{x-1}{2x}\right) \right)
\\&=-2(-x)^k(x-1)T_k\left(\frac{x-1}{2x}\right)-x^2\cdot 2(-1)^{k-1}x^{k-1}T_{k-1}\left(\frac{x-1}{2x}\right)
\\&=-(x-1)Q_k(x)-x^2Q_{k-1}(x).
\end{align*}
According to the proof of Proposition \ref{FK-convlution}, we know that $Q_k(x)$ and $L_k(x)$ satisfy the same recursion formula.
Some calculations give $Q_1(x)=-(x-1)=L_1(x)$ and $Q_2(x)=-x^2-2x+1=L_2(x)$, which completes the proof.
\end{proof}

We found an interesting fact that the triangle of coefficients of the polynomial $Q_k(x)$ is the sequence [A102587] on OEIS \cite{Sloane23}.



\section{Sulanke-Xin's Quadratic Transformation $\tau$}\label{Section-4}

According to Example \ref{Example-D33}, one can observe that the values of $(D_{K,M}(n))_{n\geq 0}$ seem to be periodically shifted. In this section, we will explain that this is not occasional.

When $F(x)$ is the unique solution of a quadratic functional equation, and $K\in \mathbb{N}$ and $M\in \mathbb{Z}$ are concrete values, we can use the \emph{Sulanke-Xin's quadratic transformation} (also called \emph{Sulanke-Xin's continued fraction method}) in \cite{Sulanke-Xin} to study the exact formula for $D_{K,M}(n)$. For instance, when $M=0$, some Hankel determinants $D_{K,0}(n)$ for convolution powers of Catalan numbers are given in \cite{Wang-Xin}; some Hankel determinants $D_{K,0}(n)$ for convolution powers of Motzkin numbers are studied in \cite{Wang-Zhang}.
Similarly, we can obtain some shifted (forward shifted for $M\geq 0$ and backward shifted for $M<0$) Hankel determinant $D_{K,M}(n)$ formulas by using Sulanke-Xin's quadratic transformation.

Now we introduce the Sulanke-Xin's quadratic transformation $\tau$.
Suppose the generating function $F(x)$ is the unique solution of a quadratic functional equation that can be written as
\begin{align}\label{xinF(x)}
F(x)=\frac{x^d}{u(x)+x^kv(x)F(x)},
\end{align}
where $u(x)$ and $v(x)$ are rational power series with nonzero constants, $d$ is a nonnegative integer, and $k$ is a positive integer.
We need the unique decomposition of $u(x)$ with respect to $d$: $u(x)=u_L(x)+x^{d+2}u_H(x)$ where $u_L(x)$ is a polynomial of degree at most $d+1$ and $u_H(x)$ is a power series.
Then \cite[Propositions 4.1 and 4.2]{Sulanke-Xin} can be summarized as follows.
\begin{prop}\label{xinu(0,i,ii)}
Let $F(x)$ be determined by \eqref{xinF(x)}. Then the quadratic transformation $\tau(F(x))$ of $F(x)$ defined as follows gives close connections between $\det \mathcal{H}(F(x))$ and $\det \mathcal{H}(\tau(F(x)))$.
\begin{enumerate}
\item[i)] If $u(0)\neq1$, then $\tau(F(x))=G(x)=u(0)F(x)$ is determined by
$$G(x)=\frac{x^d}{u(0)^{-1}u(x)+x^ku(0)^{-2}v(x)G(x)},$$
and $\det \mathcal{H}_n(\tau(F(x)))=u(0)^{n}\det \mathcal{H}_n(F(x))$;

\item[ii)] If $u(0)=1$ and $k=1$, then $\tau(F(x))=x^{-1}(G(x)-G(0))$, where $G(x)$ is determined by
$$G(x)=\frac{-v(x)-xu_L(x)u_H(x)}{u_L(x)-x^{d+2}u_H(x)-x^{d+1}G(x)},$$
and we have
$$\det \mathcal{H}_{n-d-1}(\tau(F(x)))=(-1)^{\binom{d+1}{2}}\det \mathcal{H}_n(F(x));$$

\item[iii)] If $u(0)=1$ and $k\geq2$, then $\tau(F(x))=G(x)$, where $G(x)$ is determined by
$$G(x)=\frac{-x^{k-2}v(x)-u_L(x)u_H(x)}{u_L(x)-x^{d+2}u_H(x)-x^{d+2}G(x)},$$
and we have
$$\det \mathcal{H}_{n-d-1}(\tau(F(x)))=(-1)^{\binom{d+1}{2}}\det \mathcal{H}_n(F(x)).$$
\end{enumerate}
\end{prop}

In this paper, we use $\tau$ to denote the composite step consisting of the transformation in
Proposition \ref{xinu(0,i,ii)}(i), when needed, followed by the transformation in (ii) or (iii).
Thus, if $G(x)=\tau(F(x))$, then
\[
\det  \mathcal{H}_n(F(x))=(-1)^{\binom{d+1}{2}}u(0)^{-n}\det  \mathcal{H}_{n-d-1}(G(x)),\qquad n\ge d+1.
\]

Our focus is on the general formula for $D_{3,-3}(n)$ in Example \ref{Example-D33}, as the other cases in Example \ref{Example-D33} are handled similarly.
\begin{prop}\label{Prop-D33-proof}
Let $F(x)$ be the generating function of the Motzkin numbers. Then we have
$$D_{3,-3}(0)=1,\ \ D_{3,-3}(1)=D_{3,-3}(2)=D_{3,-3}(3)=0,$$
and for $n\geq 4$,
\[ D_{3,-3}(n)=\begin{cases}
k &\ \text{if}\ \ n=3k+1; \\
0 &\ \text{if}\ \ n=3k+2; \\
-k &\ \text{if}\ \ n=3k+3. \\
\end{cases}\]
\end{prop}
\begin{proof}
Let $G(x)=x^3F(x)^3$. We can easily obtain
$D_{3,-3}(n)=\det\mathcal{H}_n(x^3F(x)^3)=\det\mathcal{H}_n(G(x))$.
We have the functional equation
\begin{align*}
G(x)=-{\frac {{x}^{3}}{G(x){x}^{3}-2\,{x}^{3}+3\,x-1}}.
\end{align*}
We apply Sulanke-Xin's continued fractions method to $G_0:=G(x)$ by repeatedly using the transformation $\tau$.
This results in a shifted periodic continued fraction of order 2:
\begin{align*}
G_0(x)\mathop{\longrightarrow}\limits^\tau G_1^{(0)}(x)\mathop{\longrightarrow}\limits^\tau G_{2}^{(0)}(x)\mathop{\longrightarrow}\limits^\tau G^{(1)}_{3}(x)\mathop{\longrightarrow}\limits^\tau G^{(1)}_{4}(x)\mathop{\longrightarrow}\limits^\tau G^{(1)}_{5}(x)= G_3^{(2)}(x)\cdots.
\end{align*}
We obtain
\begin{align}\label{e-3G0}
\det \mathcal{H}_k(G_0(x))=\det \mathcal{H}_{k-4}(G_1^{(0)}(x))=-\det \mathcal{H}_{k-6}(G_2^{(0)}(x))=-(-2)^{k-6}\det \mathcal{H}_{k-7}(G_3^{(1)}(x)).
\end{align}
For $p\geq 1$, repeat application of Proposition \ref{xinu(0,i,ii)} suggests that we shall define
\begin{align*}
  G^{(p)}_{3}(x) =&-{\frac {p^2x}{(x^5+3p(p+1)x^3-p(p+1)x^2)G^{(p)}_{3}(x)+2px^3-3p(p+1)x+p(p+1)}}.
\end{align*}
Our definition is well-motivated. The proof proceeds by mathematical induction on \( p \).

\noindent \textbf{Base Case:} For \( p = 1 \), the statement is verified by three times applications of Proposition \ref{xinu(0,i,ii)}, which yields the desired result.

\noindent \textbf{Inductive Step:} Assume the statement holds for \( p = n \). For \( p = n+1 \), we have
\begin{align*}
G^{(n+1)}_{3}(x)=&\tau (\tau( G^{(n)}_{3}(x)))
\\=&-(n+1)^2x/((x^5+3(n+2)(n+1)x^3-(n+2)(n+1)x^2)G^{(n+1)}_{3}(x)
\\ &+2(n+1)x^3-3(n+2)(n+1)x+(n+2)(n+1)).
\end{align*}
This establishes the identity for $p = n+1$, thereby completing the inductive step.

Then the results can be summarized as follows:
\begin{align*}
& \det \mathcal{H}_{k-1}\left(G_{3}^{(p)}(x)\right)=-\left(-\frac{p}{p+1}\right)^{k-1} \det \mathcal{H}_{k-3}\left(G_{4}^{(p)}(x)\right)
\\& \det \mathcal{H}_{k-3}\left(G_{4}^{(p)}(x)\right)=\left(-\frac{p+2}{p+1}\right)^{k-3} \det \mathcal{H}_{k-4}\left(G_{3}^{(p+1)}(x)\right).
\end{align*}
Combining the above formulas gives the recursion
\[\det\mathcal{H}_{k-1}\left(G_{3}^{(p)}(x)\right)=-\left(\frac{p}{p+1}\right)^{k-1}\left(\frac{p+2}{p+1}\right)^{k-3} \det \mathcal{H}_{k-4}\left(G_{3}^{(p+1)}(x)\right).\]
Let $k-1=3n+j$, where $0\leq j<3$. We then deduce that
\begin{align}\label{e-3G3n0}
\det \mathcal{H}_{3 n+j}\left(G_{3}^{(1)}(x)\right)=(-1)^{n}\left(\frac12\right)^{3n+j+1}\frac{(n+2)^{j+1}}{(n+1)^{j}} \det \mathcal{H}_{j}\left(G_{3}^{(n+1)}(x)\right).
\end{align}
Then combine \eqref{e-3G0} and \eqref{e-3G3n0}. We have
\begin{align}\label{e-3G3n}
\det \mathcal{H}_{3 n+j+7}(G_0(x))=\det \mathcal{H}_{3n+j+3}(G_1^{(0)}(x))=(-1)^{j+2}\frac{(n+2)^{j+1}}{(n+1)^{j}} \det \mathcal{H}_{j}\left(G_{3}^{(n+1)}(x)\right).
\end{align}
The initial values are
\begin{align*}
\det \mathcal{H}_{0}\left(G_{3}^{(n+1)}(x)\right)= 1; \quad \det \mathcal{H}_{1}\left(G_{3}^{(n+1)}(x)\right)=0 ; \quad \det \mathcal{H}_{2}\left(G_{3}^{(n+1)}(x)\right)=-\left(\frac{n+1}{n+2}\right)^{2}.
\end{align*}
Using the above initial values, \eqref{e-3G0}, and \eqref{e-3G3n}, we obtain
\begin{align*}
\det \mathcal{H}_{0}(G(x))&=1,\ \  \det \mathcal{H}_{1}(G(x))=\det \mathcal{H}_{2}(G(x))=\det \mathcal{H}_{3}(G(x))=0,
\end{align*}
and for $n \geq 1$,
\begin{align*}
\det \mathcal{H}_{3n+1}(G(x))&=-\det \mathcal{H}_{3n+3}(G(x))=n,\ \ \ \det \mathcal{H}_{3n+2}(G(x))=0.
\end{align*}
This completes the proof.
\end{proof}

We can also obtain the following results.
\begin{prop}\label{Prop-Mozt-123}
Let $F(x)$ be the generating function of the Motzkin numbers. Then we have
\[ D_{2,-1}(n)=\begin{cases}
1 &\ \text{if}\ \ n\equiv 0 \mod 4; \\
0 &\ \text{if}\ \ n\equiv 1,3 \mod 4; \\
-1 &\ \text{if}\ \ n\equiv 2 \mod 4. \\
\end{cases}\]
and
\[ D_{2,-2}(n)=\begin{cases}
1 &\ \text{if}\ \ n\equiv 0,9,10,11 \mod 12; \\
0 &\ \text{if}\ \ n\equiv 1,2,7,8 \mod 12; \\
-1 &\ \text{if}\ \ n\equiv 3,4,5,6 \mod 12. \\
\end{cases}\]
For $n\geq 4$, we get
\[ D_{2,-3}(n)=\begin{cases}
1 &\ \text{if}\ \ n\equiv 0,4 \mod 12; \\
0 &\ \text{if}\ \ n\equiv 2,8 \mod 12; \\
-1 &\ \text{if}\ \ n\equiv 6,10 \mod 12; \\
8k &\ \text{if}\ \ n=12k+1; \\
-8k &\ \text{if}\ \ n=12k+3; \\
8k+2 &\ \text{if}\ \ n=12k+5; \\
-8k-4 &\ \text{if}\ \ n=12k+7; \\
8k+4 &\ \text{if}\ \ n=12k+9; \\
-8k-6 &\ \text{if}\ \ n=12k+11.
\end{cases}\]
\end{prop}

The proof of Proposition \ref{Prop-Mozt-123} is analogous to Proposition \ref{Prop-D33-proof}, and thus we omit it here.
Note that some of the formulas in Proposition \ref{Prop-Mozt-123} can also be obtained from Corollary \ref{Cor-Motzkin}, for example: When $K=2$ and $m=0$, it yields $D_{2,-1}(n+2)=-D_{2,-1}(n)$; When $K=2$ and $m=1$, it yields $D_{2,-2}(n+3)=-D_{2,0}(n)$.
Here, however, we would like to mention an ingenious application of the Sulanke--Xin's quadratic transformation.
Similarly, when $F(x)$ are the generating functions of Motzkin numbers, we can study the exact formula for $D_{K,M}(n)$ for some concrete $K\in \mathbb{N}$ and $M\in \mathbb{Z}$ by using Sulanke--Xin's quadratic transformation.
We will not list them one by one here, because they can be performed in an entirely similar manner as in \cite{Wang-Xin,Wang-Zhang}.

\section{Concluding Remarks}\label{Section-5}

In this paper, we establish the identity (Proposition \ref{FK-convlution}) about the $k$-th convolution of a certain class of power series $F(x)$.
Then, by employing this identity, we extend Cigler's determinant identities (Theorem \ref{Cigler-Catalan}) to the convolution of general power series $F(x)$, where $F(x)$ satisfies a certain type of quadratic functional equation, see Theorem \ref{Main-Theorem}. Finally, we apply this result to generating functions of Catalan numbers, Motzkin numbers, etc.
As byproducts, we obtain some precise determinant formulas through Sulanke-Xin's quadratic transformation.

Fulmek \cite{Fulmek-2402} gave a bijective proof of Cigler's conjecture by interpreting determinants as generating functions of nonintersecting lattice paths. Inspired by this work, we raise the following question.
\begin{Openprob}
Let $F(x)$ be the generating function of the Motzkin numbers. Find a bijective proof for Corollary \ref{Cor-Motzkin}.
\end{Openprob}




\noindent\textbf{Statements and Declarations:} The authors declared that they had no conflicts of interest with respect to their authorship or the publication of this article.



\noindent
{\small \textbf{Acknowledgments:}
The authors would like to show their sincere appreciations to all suggestions improving the presentation of this paper.
The authors are greatly indebted to Professor Guoce Xin for the guidance over the past years.
Feihu Liu was partially supported by the Postdoctoral Fellowship Program and China Postdoctoral Science Foundation (Grant No. BX2026002). Ying Wang was partially supported by the Natural Science Foundation of Henan Province (Grant No. 252300420912).


\begin{thebibliography}{99}

\bibitem{Andrews-Wimp} G. Andrews and J. Wimp, \emph{Some $q$-orthogonal polynomials and related Hankel determinants}, Rocky Mountain J. Math. 32(2) (2002), 429--442.

\bibitem{chien2022hankel} H.-L. Chien, S.-P. Eu, and T.-S. Fu, \emph{On Hankel determinants for Dyck paths with peaks avoiding multiple
  classes of heights}, European J. Combin, 101 (2022), 103478.

\bibitem{Cigler-1302} J. Cigler, \emph{A special class of Hankel determinants}, arXiv:1302.4235, (2013).

\bibitem{Cigler-2306} J. Cigler, \emph{Shifted hankel determinants of Catalan numbers and related results II: Backward shifts}, arXiv:2306.07733, (2023).

\bibitem{Cigler-2308} J. Cigler, \emph{Some experimental observations about Hankel determinants of convolution powers of Catalan numbers}, arXiv:2308.07642, (2023).

\bibitem{Cigler-2403} J. Cigler, \emph{Hankel determinants of convolution powers of Catalan numbers revisited}, arXiv:2403.11244, (2024).

\bibitem{cigler2011some} J. Cigler and C. Krattenthaler, \emph{Some determinants of path generating functions}, Adv. in Appl. Math. 46 (2011), 144--174.

\bibitem{Elouafi} M. Elouafi, \emph{A unified approach for the Hankel determinants of classical combinatorial numbers}, J. Math. Anal. Appl. 431 (2015), 1253--1274.

\bibitem{FlajoletDM} P. Flajolet, \emph{Combinatorial aspects of continued fractions}, Discrete Math. 32 (1980), 125--161.

\bibitem{Fulmek-2402} M. Fulmek, \emph{Hankel determinants of convoluted Catalan numbers and nonintersecting lattice paths: A bijective proof of Cigler's conjecture}, arXiv:2402.19127v2, (2024).

\bibitem{GesselViennot} I. M. Gessel and G. Viennot, \emph{Determinants, paths, and plane partitions}, preprint. (1989), available electronically at http://people.brandeis.edu/$\sim$gessel.

\bibitem{Gessel-Xin} I. M. Gessel and G. Xin, \emph{The generating function of ternary trees and continued fractions}, Electron. J. Combin. 13 (2006), \#R53.

\bibitem{HanGuoNiu} G.-N. Han, \emph{Hankel continued fraction and its applications}, Adv. Math. 303 (2016), 295--321.

\bibitem{QH-Hou} Q.-H. Hou, A. Lascoux, and Y.-P. Mu, \emph{Evaluation of some Hankel determinants}, Adv. Appl. Math. 34 (2005), 845--852.

\bibitem{Krattenthaler05} C. Krattenthaler, \emph{Advanced determinant calculus: a complement}, Linear Algebra Appl. 411 (2005), 68--166.

\bibitem{Lindstrom} B. Lindstr\"om, \emph{On the vector representations of induced matroids}, Bull. London Math. Soc. 5 (1973), 85--90.

\bibitem{Mu-Wang} L. Mu and Y. Wang, \emph{Hankel determinants of shifted Catalan-like numbers}, Discrete. Math. 340 (2017), 1389--1396.

\bibitem{Mu-Wang-Yeh} L. Mu, Y. Wang, and Y.-N. Yeh, \emph{Hankel determinants of linear combinations of consecutive Catalan-like numbers}, Discrete. Math. 340 (2017), 3097--3103.

\bibitem{Sloane23} N. J. A. Sloane, \emph{The on-line encyclopedia of integer sequences}, published electronically at http://oeis.org. (2025).

\bibitem{RP.Stanley99} R. P. Stanley, \emph{Enumerative Combinatorics (volume 2)}, second Edition. Cambridge Studies in Advanced Mathematics, vol. 208, Cambridge University Press, (2023).

\bibitem{Sulanke-Xin} R. A. Sulanke and G. Xin, \emph{Hankel determinants for some common lattice paths}, Adv. in Appl. Math. 40(2) (2008), 149--167.

\bibitem{Wang-Xin} Y. Wang and G. Xin, \emph{Hankel determinants for convolution powers of Catalan numbers}, Discrete. Math. 342 (2019), 2694--2716.

\bibitem{Wang-Xin-Zhai} Y. Wang, G. Xin, and M. Zhai, \emph{Hankel determinants and shifted periodic continued fractions}, Adv. Appl. Math. 102 (2019), 83--112.

\bibitem{Wang-Zhang} Y. Wang and Y. Zhang, \emph{Hankel determinants for convolution powers of Motzkin numbers}, arXiv:2502.21050, (2025).

\bibitem{Zeilberger} D. Zeilberger, \emph{The method of creative telescoping}, J. Symbolic Comput. 11 (1991), 195--204.

\end{thebibliography}
\end{document}